\documentclass[11pt]{article}
\usepackage{amsmath}
\usepackage{amssymb}
\usepackage{amsthm}
\usepackage{graphicx}

\usepackage{graphicx}
\usepackage{amsmath}
\usepackage{amsthm}
\usepackage{amsfonts}
\usepackage{amssymb, amscd}

\usepackage{euscript}
\usepackage[all,tips]{xy}
\usepackage{epic,eepic}

\usepackage{color}

\newcommand{\Z}{\mathbb Z}

\newcommand{\g}{\mathfrak{g}}

\newtheorem{theorem}{Theorem}[section]
\newtheorem{lemma}[theorem]{Lemma}

\newtheorem{proposition}[theorem]{Proposition}
\newenvironment{definition}[1][Definition.]{\begin{trivlist}
\item[\hskip \labelsep {\bfseries #1}]}{\end{trivlist}}
\newenvironment{example}[1][Example]{\begin{trivlist}
\item[\hskip \labelsep {\bfseries #1}]}{\end{trivlist}}

\begin{document}

\def\C{{\mathbb C}}
\def\N{{\mathbb N}}
\def\Z{{\mathbb Z}}
\def\R{{\mathbb R}}
\def\PP{\cal P}
\def\p{\rho}
\def\phi{\varphi}
\def\ee{\epsilon}
\def\ll{\lambda}
\def\l{\lambda}
\def\a{\alpha}
\def\b{\beta}
\def\D{\Delta}
\def\g{\gamma}
\def\rk{\text{\rm rk}\,}
\def\dim{\text{\rm dim}\,}
\def\ker{\text{\rm ker}\,}
\def\square{\vrule height6pt width6pt depth 0pt}
\def\epsilon{\varepsilon}
\def\phi{\varphi}
\def\kappa{\varkappa}
\def\strl#1{\mathop{\hbox{$\,\leftarrow\,$}}\limits^{#1}}

\vskip1cm

\title{On Koszulity for operads of Conformal Field Theory}



\vskip1cm

\author{
 Natalia Iyudu,  Abdenacer Makhlouf}

\date{}

\maketitle







{\centerline{\bf Abstract}}

\bigskip

{\small

We study  two closely related operads: the Gelfand-Dorfman
operad ${\cal G}D$ and the
Conformal  Lie Operad ${\cal CL}ie$. The latter is the operad governing the Lie conformal  algebra structure.
We prove Koszulity of the Conformal Lie operad  using the Gr\"obner bases theory for operads and
an operadic analogue of the Priddy criterion.
An example of deformation of an operad  coming from the Hom structures is considered.
In particular we study possible deformations of the Associative operad from the point of view of the confluence property. Only one deformation, the operad which governs the identity $(\a(ab))c=a(\a (bc))$
turns out to be confluent.
We introduce a new Hom structure, namely Hom--Gelfand-Dorfman algebras and study their basic properties.

}
\bigskip

\normalsize

\section{Introduction}
\vskip11mm

One of  the key notions of the two-dimensional Conformal Field Theory
is the  operator product expansion (OPE) of chiral quantum fields
(see e.g. \cite{Kac}
):
$$
\Phi(x)\Psi(y)= \sum_{n\in \Z}\frac{(\Phi_{(n)}\Psi)(y)}
{(x-y)^{n+1}}.
$$
This defines a binary operation
$\Phi_{(n)}\Psi$
on the space of chiral fields for each
$n\in \Z$. These binary operations, along with the identity field and
the unary operation $\partial=\partial/\partial y$ form a structure,
called a vertex
(or chiral) algebra, which has been axiomatized by Borcherds \cite{Bor}.
The most important part of the OPE is its singular part since it allows to
compute commutators of chiral fields. This involves only the binary operations
$\Phi_{(n)}\Psi$ for each $n\in \Z_+$, and along with the unary operation
$\partial$, forms a structure, called a Lie conformal algebra, which has been
axiomatized by Kac \cite{Kac}.

{\bf Definition.} A {\it Lie conformal algebra} is a $\mathbb{C}[\partial]$ -
module $R$ endowed with binary bilinear over $\mathbb{C}$ operations
$\,\,\,{(n)} \,\,: R \times R \to R$ for each $n\in \Z_+$,
which satisfy the following axioms:
$$
(c1) \,\,  a_{(n)} b = 0, \,\, n \gg0,
$$
$$
(c2) \,\, (\partial a)_{(n)} b = -n a_{(n-1)} b, \,\,  a_{(n)}\partial b =
\partial(a_{(n)} b )+ n a_{(n-1)} b,
$$
$$
(c3) \,\, a_{(n)} b = - \sum\limits_{j=0}^{\infty} (-1)^{(n+j)} \partial^{(j)}
(b_{(n+j)} a),
$$
$$(c4) \,\, a_{(m)} (b_{(n)}c) - b_{(n)} (a_{(m)}c) =
\sum\limits_{j=0}^{m}\left({m \atop j}\right) (a_{(j)} b)_{(m+n-j)} c.
$$
In the axiom (c3) and further on we use the divided power notation:
$\partial^{(n)}=\partial^n/n!$.

Lie conformal algebras occur naturally in Classical Field Theory as well.
Namely, consider a linear local Poisson bracket
$$
\{u_i(x),u_j(y)\}= H_{ij}(u(y),u'(y),\cdots,u^{(s)}(y);\partial/\partial y)\delta(x-y),
$$
where $H=(H_{ij}(\partial/\partial y))$ is a matrix differential operator
of finite order,
linear in the $u_i$ and their derivatives, and let
$H_{ij}(\lambda)=\sum\limits_{s\in \Z} H_{ij}^s \frac{\lambda^s}{s!}$ be its symbol.
Then the $\C[\partial]$ - module, consisting of linear combinations of functions $u_i$
and their derivatives, is a Lie conformal algebra, endowed with
the products ${u_i}_{(s)}u_j =H_{ji}^s,\,s\in \Z_+$.

One of  the simplest examples of a linear local Poisson bracket, when $H$
has order one, was studied by Gelfand and Dorfman \cite{GD}; it is defined by the following symbol:
$$
H_{ij}(\lambda)=\sum_s c_{ij}^s u_s +
\sum_s d_{ij}^s u'_s +\lambda \left(\sum_s (d_{ij}^s + d_{ji}^s) u_s \right),
$$
where
$c_{ij}^s$ and $d_{ij}^s$
are constants, satisfying certain quadratic
relations. In order to explain these relations, consider the $\C$-span
$V$ of the $u_i$, and denote by
$[,]$ and $\circ$ the bilinear products,
with the structure constants $c_{ij}^s$ and $d_{ij}^s$ respectively.
Then
$(H_{ij}(\lambda))$ defines a local Poisson bracket if and only if
$[,]$ is a Lie algebra bracket, $\circ$ is a left symmetric algebra,
such that the operators of right multiplications commute,
and these two binary operations satisfy the compatibility condition (5) below.
Such a structure is called a Gelfand-Dorfman algebra. The special case,
when the bracket is zero, is called a Novikov algebra \cite{Nov}, \cite{Z}.

We consider in this paper  operads, which describe the structure of the
Gelfand-Dorfman algebras,  and of the Lie conformal algebras.
We denote  by ${\cal CL}ie$ the operad, which controls the structure of
Lie conformal
algebras, and we call it the ${\cal CL}ie$ (conformal Lie) operad.
In Theorem \ref{CLKosz} (Section \ref{CL}) we prove that  ${\cal CL}ie$ is Koszul, using the Gr\"obner bases theory  (Diamond Lemma) for operads and an operadic analogue of the Priddy criterion.
An explanation of these tools will be given in Section \ref{Gr}.

The complete list of relations for operations  $\circ$ and $[,]$,
defining a left(right) Gelfand-Dorfman algebra is the following:

\vspace{5mm}

$(1) \,\, (a \circ b)\circ c-a \circ (b\circ c)=(b\circ a)\circ c-b\circ(a\circ c) \,\,$ (left-symmetric),

\vspace{2mm}

$(2) \,\, (a\circ b)\circ c=(a\circ c)\circ b \,\,$  (right Novikov),

\vspace{2mm}

$(3) \,\, [a,b]=-[b,a] \,\,$,

\vspace{2mm}

$(4) \,\, [[a,b],c]+[[c,a],b]+[[b,c],a]=0 \,\,$,

\vspace{2mm}

$(5) \,\, [c \circ a,b]-[c\circ b,a]+[c,a]\circ b-[c,b]\circ a-c\circ [a,b]=0 \,\,$.

\vspace{5mm}

We call an algebra satisfying these relations a {\it right} GD-algebra. The algebra with opposite multiplication to  $\circ$, will be called a {\it left} GD-algebra.

%
%

In Section \ref{Def} we touch upon the question of the deformation theory for operads.
Well-known classes of algebras, such as Associative, Lie, Jordan, Novikov, Nambu, Poisson, Gerstenhaber, Batalin-Vilkovisky, etc. are governed by corresponding operads, they consist of algebras over those operads.

We study several examples of deformation of operads. We emphasize that algebras over those deformed operads are Hom--algebras, defined in the extensive literature over the last years \cite{LS1, MS, HC, Yau,YauP}.
This notion should be distinguished from the notion of a deformation of an algebra over a given operad,
developed and heavily used before in varying generality \cite{K, B, Al}.
We use here a naive version of a deformation of an operad, in order to provide examples which could lead to the general theory. Namely, by a deformation of an operad we mean here the following.

Consider an operad $\cal A$, and introduce an extra generator $h$ in the presentation of this operad by generators and relations. An extra operator corresponds in the algebra representation to a unary operation $\alpha: A \to A$. Then we write down relations, involving $h$, of the deformed operad ${\cal A}_h$ in  a certain way, such that if $h=id$, we get the old relations. This new operad ${\cal A}_h$ can be considered as a deformation of $\cal A$.

The examples of deformations we consider here govern exactly the notion of Hom--structures. For example, for operad ${\cal A}ss$, defined by one generation operation $m$ and one relation:

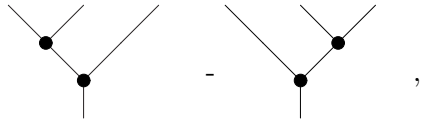
\begin{center}
\setlength{\unitlength}{1.0mm}
\begin{picture}
(25,15)(5,5)
\drawline(15,0)(15,5)
\drawline(15,5)(20,10)
\drawline(15,5)(10,10)
\drawline(20,10)(25,15)
\drawline(10,10)(5,15)
\drawline(10,10)(15,15)
\put(15,5){\circle*{1.7}}
\put(10,10){\circle*{1.7}}
\end{picture}
-
\begin{picture}
(25,15)(5,5)
\drawline(15,0)(15,5)
\drawline(15,5)(20,10)
\drawline(15,5)(10,10)
\drawline(20,10)(25,15)
\drawline(20,10)(15,15)
\drawline(10,10)(5,15)
\put(15,5){\circle*{1.7}}
\put(20,10){\circle*{1.7}}
\end{picture},
\end{center}
\vskip5mm
consider a deformed operad ${\cal A}ss_h$, defined by two generators, a binary operation $m$, and  a unary operation $h$, and a relation

 \begin{center}
\setlength{\unitlength}{1.0mm}
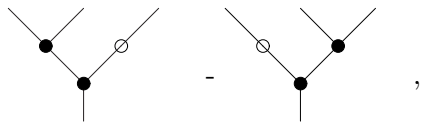
\begin{picture}
(25,15)(5,5)
\drawline(15,0)(15,5)
\drawline(15,5)(20,10)
\drawline(15,5)(10,10)
\drawline(20,10)(25,15)
\drawline(10,10)(5,15)
\drawline(10,10)(15,15)
\put(15,5){\circle*{1.7}}
\put(20,10){\circle{1.7}}
\put(10,10){\circle*{1.7}}
\end{picture}
-
\begin{picture}
(25,15)(5,5)
\drawline(15,0)(15,5)
\drawline(15,5)(20,10)
\drawline(15,5)(10,10)
\drawline(20,10)(25,15)
\drawline(20,10)(15,15)
\drawline(10,10)(5,15)
\put(15,5){\circle*{1.7}}
\put(20,10){\circle*{1.7}}
\put(10,10){\circle{1.7}}
\end{picture},
\end{center}

 \vspace{5 mm}



 The operation $h$ denoted here by an empty dot and $m$ by a filled one. An algebra over the operad
${\cal A}ss_h$ is a Hom-associative algebra defined in \cite{MS}.

We therefore study here examples of deformations of operads inspired by Hom--structures.

We consider all operads which can be obtained from ${\cal A}ss$ by deformations, of this kind, introducing an extra generator --- a unary operation. We classify them from the point of view of the confluence property of the deformed operad. Application of the version of the Gr\"obner bases theory for  nonsymmetric operads (which allows to deal with operations of arbitrary arity),  leads to the result, that only
one of those deformations is an operad with {\it confluent} presentation. This is the operad which governs the identity $(\alpha(ab))c=a(\alpha(bc)) $.
In particular, this operad is Koszul.

In  Section \ref{HGD} we introduce the notion of a Hom--Gelfand-Dorfman (Hom--GD) algebra. Namely we suggest to study a certain particular deformation of Gelfand-Dorfman operad.
As it was explained above the Gelfand-Dorfman structure is an enriched version of the Novikov structure, corresponding to more general type of local Poisson bracket. The latter have been extensively studied  (see for example \cite{Z},  \cite{Burde} and references therein), including the Hom versions \cite{Yau}.
We establish a number of basic properties of the Hom--GD algebra, the algebra over a certain deformation of the operad  {\cal GD}.


\section{Gr\"obner bases theory for operads}\label{Gr}

In this section we remind the definition of the (symmetric, shuffle, nonsymmetric) operad and set up the basics of the Gr\"obner bases theory for operads, analogous to the Gr\"obner bases theory for algebras presented by generators and relations. For more detailed introduction to various aspects of operads, including those discussed here see \cite{Bruno} and \cite{MM}.

The main ingredient for the development of the analogue of the Gr\"obner bases theory for  algebras in the operadic setting is the notion of shuffles and their orderings, which appeared first in this context in the paper due to Hoffbeck \cite{Hof}.
After introducing shuffles one can deal with elements of the free symmetric operad (trees with marked leaves and vertices) exactly in the same manner, as one can do with monomials in the free associative algebra.
This was written down precisely in the paper due to Dotsenko and Khoroshkin \cite{Dots}.

To explain main points of the Gr\"obner bases theory for operads presented by generators and relations,
where elements of the free operad are depicted  by tree-monomials, we will need to clarify how to order tree-monomials, and what means that one tree-monomial divides another. On the basis of this the whole
Gr\"obner bases theory works in a very much parallel way, as in the case of algebras.

Let us start with the definition of an operad.
The notion of an operad can be made very general, depending on the category, where we prefer to work.
The key point is to define the 'product', which makes a category of 'collections' into a monoidal category.

The first basic notion is a collection.
\begin{definition}  We define a {\it nonsymmetric collection}
as a functor from the category of non-empty finite ordered sets {\it Ord} (with order-preserving bijections as morphisms) to the category of vector spaces {\it Vect}.
\end{definition}

In this case a nonsymmetric collection is just the same as a graded vector space.

However to get a more general definition one can substitute the category of vector spaces by graded vector spaces, or dg-vector spaces, i.e. chain complexes, which is done for objects appearing in, say, algebraic topology.

\begin{definition}  We define a {\it symmetric collection}
as a functor from the category of non-empty finite sets {\it Fin} (with  bijections as morphisms) to the category of vector spaces {\it Vect}.
\end{definition}

In this case a symmetric collection is also called a 'S-module', which means it is a set of right $k[S_n]$-modules $M(n)$ for any $ n \in {\mathbb Z}_+$.

Note, that it can be shown that all vector spaces ${\cal P}(I)$, for $I \in $ {\it Fin} are defined by ${\cal P}(n)$=${\cal P}(\{\bar{1,n}\})$, so that $ {\cal P}(n)$, $n \in {\mathbb Z}_+$ is usually what is meant by a collection.

Now we can define a natural compositions of these collections.

\begin{definition} Let ${\cal P}$ and ${\cal Q }$ be two nonsymmetric collections. Their {\it nonsymmetric composition} is defined by the formula.

$$
{\cal P} \circ {\cal Q} (I) = \bigoplus_k {\cal P}(k) \otimes \,\,( \bigoplus_{f:I \to \{\bar{1,k}\}}
{\cal Q} (f^{-1}(1)) \otimes \cdots \otimes {\cal Q} (f^{-1}(k)) \,\,)
$$

Here the sum is taken over all monotonically non-decreasing surjections $f$.

\end{definition}

\begin{definition} Let ${\cal P}$ and ${\cal Q }$ be two nonsymmetric collections. Their {\it shuffle composition} is defined by the formula.

$$
{\cal P} \circ_{sh} {\cal Q} (I) = \bigoplus_k {\cal P}(k) \otimes \,\, ( \bigoplus_{f:I \to \{\bar{1,k}\}}
{\cal Q} (f^{-1}(1)) \otimes \cdots \otimes {\cal Q} (f^{-1}(k))\,\,)
$$

Here the sum is taken over all shuffling surjections $f$, i.e. those for which $f^{-1}(i) < f^{-1}(j)$, if $i<j$.

\end{definition}

\begin{definition} Let ${\cal P}$ and ${\cal Q }$ be two symmetric collections. Their {\it symmetric composition} is defined by the formula

$$
{\cal P} \circ {\cal Q} (I) = \bigoplus_k {\cal P}(k) \otimes_{k[S_k]} \,\, ( \bigoplus_{f:I \to \{\bar{1,k}\}}
{\cal Q} (f^{-1}(1)) \otimes \cdots \otimes {\cal Q} (f^{-1}(k))\,\,)
$$

Here the sum is taken over all  surjections $f$.

\end{definition}

It can be proven that defined above composition (nonsymmetric, shuffle, symmetric) makes the set of corresponding (nonsymmetric or symmetric) collections into a monoidal category, i.e.
 the pentagon and triangular diagrams are satisfied, and the proper unital map is present.

Now any monoid in the corresponding monoidal category we call a {\it (nonsymmetric, shuffle, symmetric) operad.}

\begin{example}
Let $V$ be a vector space, denote by $\rm{End}_V$ the {\it operad  of multilinear mappings}. It is a collection $Hom(V^{\otimes n}, V), \, n \ge 0$ of all multilinear mappings, with the composition maps, defined in a natural way.
\end{example}

\begin{definition} An algebra over an operad ${\cal P}$ is a vector space $V$ together with an operadic morphism from ${\cal P}$ to a corresponding operad of multilinear mappings ${\rm End}_V$.
\end{definition}

The 'building blocks' for the compositions we just defined (which serves as a monoidal product) are the following {\it elementary compositions}.
We will need them later for the Gr\"obner bases theory for symmetric operads. The case of symmetric operads is what we mainly will need for our results, so we deal with this version of the Gr\"obner bases theory.  The nonsymmetric case is easier, and could be obtained as a particular case of the symmetric one.

\begin{definition} We say that an element $a$ of an operad ${\cal P}$ has an {\it arity} $n$, if it belongs to
${\cal P}(n+1)$.
\end{definition}

\begin{definition}
Let ${\cal P}$ be a symmetric operad, for elements $a \in {\cal P}(n)$ and  $b \in {\cal P}(m)$
we define

1. a {\it nonsymmetric elementary composition} $a \circ_i b$ as the operation

$$ a(x_1,\cdots,x_{i-1}, b(x_i,x_{i+1},\cdots,x_{i+m-1}),x_{i+m},\cdots,x_{m+n-1})$$

2. a {\it shuffle elementary composition} $a \circ_{i,\sigma} b$ as the operation

$$ a(x_1,\cdots,x_{i-1}, b(x_i,x_{\sigma(i+1)},\cdots,x_{\sigma(i+m-1)}),x_{\sigma(i+m)},\cdots,x_{\sigma(m+n-1)})$$

where the bijection $\sigma: \{i+1,\cdots,m+n-1\} \to \{i+1,\cdots,m+n-1\}$ is a $(m-1,n-i)$-shuffle,
which means:

$$\sigma(i+1)<\cdots<\sigma(m+n-1)$$

and

$$\sigma(i+1)<\cdots<\sigma(m+n-1).$$

3. a {\it symmetric elementary composition} $a \circ_{i,\sigma} b$ as the operation

$$ a(x_1,\cdots,x_{i-1}, b(x_i,x_{\sigma(i+1)},\cdots,x_{\sigma(i+m-1)}),x_{\sigma(i+m)},\cdots,x_{\sigma(m+n-1)})$$

where $\sigma $ ia an arbitrary permutation from $S_{m+n-1}$.

\end{definition}

Consideration of a shuffle operad instead of a symmetric operad does not lead to any loss of information,
it just corresponds to a certain choice of representatives. If we present elements of the free symmetric operad  as  decorated trees in the space, then to take a shuffle representative of that tree, means to put it into the plane in a certain, canonical way (as it is usually done with trees). The resulting plane tree will be an element of the free shuffle operad.

 In a more rigorous  way, it could be formulated as follows: for two symmetric collections $\cal P $ and $\cal Q$, for a forgetful functor from symmetric to shuffle categories of collections, the following holds:

$$({\cal P} \circ {\cal Q})^F = {\cal P}^F \circ_{sh} {\cal Q}^F,$$
that is, the forgetful functor is monoidal.

So from now on our main object will be the free shuffle operad.
We will use the well-known presentation of elements of the free operad by {\it decorated trees}.
Our  rooted tree is a connected graph of genus $0$, whose vertices are labeled (by operations).
Each vertex should have at least one input and exactly one output, corresponding to edges. The ends of edges, which are not vertices are called {\it leaves}, except for one. There should be exactly one edge, which has neither a vertex nor  a leave at one end, but instead  a distinguished output vertex, called the {\it root} of the tree.
There are obvious simplest trees: the {\it degenerate tree}, containing no vertices, and {\it corolla},
containing only one vertex.

If we want to put a tree on the plane, we just fix an increasing (from left to right) order for the inputs in each vertex.
It is a usual procedure to order inputs of a given vertex $i$ as follows: we compare the minimal leaves,
one can reach from the given input, if this leave is smaller, the input is smaller. They are situated in the plane accordingly. This particular way of putting the tree on the plane  corresponds to the choice of a shuffle representative.

We will think of the free shuffle operad $ \cal F_{\cal P}$, generated by a symmetric collection
$\cal P$, as a set of planar representatives of trees corresponding to basic (corolla) operations of the collection (in each arity ${\cal P} (n)$), and all trees obtained from them by elementary shuffle compositions. The latter can be defined in terms of drafting and relabeling the trees.
We call those planar representatives shuffle trees {\it tree-monomials}.

\begin{definition} We say that the {\it degree} of a tree-monomial is a number of operations (vertices) involved in it.

\end{definition}

Now we can define the first essential for the Gr\"obner bases theory notion of {\it divisibility } for the tree-monomials.

\begin{definition} A tree-monomial $w$ is divisible by a tree-monomial $v$, if the underlying tree for $w$ (non-labeled) contains as a (non-labeled) subtree an underlying tree of $v$, and decorations are compatible in the following way: the vertices which are common for both trees $w$ and $v$ have the same labels and the leaves of $v$ are decorated according to the following rule: first is a leave through which we can get to the smallest leave of the tree $w$, second -- to the second smallest, and so on. If $v$ has such labeling, it divides $w$.
\end{definition}

Since all tree-monomials of the free operad  are obtained from corollas by elementary shuffle compositions, it is easy to see that $w$ can be obtained form $v$ by a sequence of elementary shuffle compositions.

Here we can see a difference with the notion of divisibility for monomials in the free associative algebra, which are 'one-dimensional': $w$ is divisible by $v$ means, there exists $r,l$, such that $w=lvr$. Shuffle trees are 'two-dimensional', so composition could be applied not only from the left, or from the right.

Next important step is to  define an ordering on the tree-monomials, which is compatible with  shuffle compositions.

More precisely, we say, that the ordering on tree-monomials from ${\cal F}_{\cal P}$ is {\it admissible}, if

first, the elements of smaller arity are smaller:

$$a \in {\cal F}_{\cal P}(n), b \in {\cal F}_{\cal P}(m), n<m \Longrightarrow a<b \footnote{ Note that this condition could be omitted, provided the d.c.c. property holds for the ordering. In fact, even milder restriction on the ordering  is  actually needed. For example, if relations are homogeneous, then the ordering, where elements of bigger arity are smaller is also  admissible }$$

and if for $a,a' \in {\cal F}_{\cal P}(n)\,$ and $\,b,b' \in {\cal F}_{\cal P}(m)$

$$ a \leq a' \,\, {\rm and} \,\,\, b \leq b',$$

then for all elementary shuffle compositions

$$ a \circ_{i,\sigma} b \leq a' \circ_{i,\sigma} b'.$$

We give here two examples of ordering, obeying this  property.

I.{\it Path-lexicographic ordering}

To explain this ordering we encode a tree by a certain sequence consisting of words in an alphabet corresponding to  operations, and one permutation of the labels on the leaves.
Let $a$ be a tree-monomial with $n$ leaves. Associate to $a$ a sequence of words $(u_1,\cdots,u_n)$, obtained in the following way: for the leave marked by $i$, the word $u_i$ consists of the sequence
(of operations), marking those vertices of the tree, one has to pass going from the root to the leave marked $i$. It is, of course, unique. However the information contained in these sequences is still not enough to distinguish tree-monomials. So we add an extra bit to the sequence:  a permutation $\sigma \in S_n$, which lists the labels on the leaves of the tree-monomial (according to the definition it is a planar tree, so we can list  leaves from left to right).
The described map from tree-monomials to sequences of the type $(u_1,\cdots,u_n, \sigma)$ is an injection.

Hence if we define how to order such sequences, it will give us an ordering on the tree-monomials.
The order on the sequences will be degree-lexicographical:

first, the longer sequence is bigger: $(u_1,\cdots,u_n, \sigma)>(u'_1,\cdots,u'_m, \sigma')$ if $n>m$;

if sequences have the same length, we compare sequences element-wise,
using the following order on words and permutations.

Words are compared degree-lexicographically (from left to right).

If all words coincide, we compare permutations, for them we use reverse lexicographical order:
 suppose for two permutations $\sigma$ and $\sigma'$ the first from the left position, where entries $\sigma_i$ and $\sigma'_i$ are different is $i$, then if  $\sigma_i > \sigma'_i$, we say that $\sigma < \sigma'$.

 This ordering is admissible. Of course it allows slight variations, we can use the reverse lexicographic order instead of the usual one, compare words from right to left, and so on.

 The choice of  a correct ordering might be very essential for the conclusions one can deduce from the Gr\"obner bases computations.

Let us describe here also another type of possible ordering.

II.{\it Forest-lexicographic ordering}

Let $a $ and $b$ be  tree-monomials with leaves labeled by  arbitrary finite subsets $I, I' \subset \mathbb N$ respectively. Consider first operations at the roots of $a$ and $b$, which we call $A$ and $B$ respectively. We have $k$ and $l$ trees, where $k$ and $l$ defined by arities of the operations $A$ and $B$, starting at $A$ and $B$, so:

$$
a=A(a_1,\cdots,a_k); \, b=B(b_1,\cdots,b_l)
$$

Leaves of the trees $a$ and $b$ become  disjoint unions of the sets of leaves of trees
$a_1,\cdots,a_k$ and $b_1,\cdots,b_l$ which we denote $I_j \subset I$ and  $I'_j \subset I'$ respectively.
Thus,

$$ I = I_1 \sqcup \cdots \sqcup I_k, \, I' = I'_1 \sqcup \cdots \sqcup I'_l, $$
where, as usual, for tree-monomoials we have $ {\rm min} I_1 < \cdots < {\rm min} I_k$ and
$ {\rm min} I'_1 < \cdots < {\rm min} I'_l.$

We compare $a$ and $b$ as follows.
First, we compare sets $I$ and $I'$ of leaves of those trees.
If $I$ and $I'$ are arbitrary finite subsets of natural numbers, to compare them we first order the numbers:

$$ I=\{i_1<\cdots<i_k\}, \, I'=\{i'_1<\cdots<i'_l\},$$

and then order them (reverse)lexicographically: suppose $1<2<\cdots<n$, and if  $j$th place is the first (from the left) where $i_j < i'_j$ then $I > I'$.
If $I < I'$ then $a < b$.

When $I=I'$, we compare $A$ and $B$: if $A<B$, then $a<b$.

In the case  $I=I'$ and $A=B$, we say that $a<b$ if:
for the smallest $j$, such that pairs $(I_j, a_j) \neq (I'_j, b_j)$ do not coincide, either
$ I_j < I'_J$ or $ I_j = I'_J$ and $a_j<b_j$.
It is a recursive definition of the ordering.

\vskip5 mm

Since we have now the ordered monoid of tree-monomials with shuffle composition, as a multiplication, compatible with the ordering, we can define {\it reduction} and proceed with the definition of
Gr\"obner basis. Then, after defining the concept of {\it ambiguity} and {\it s-polynomial} we can formulate a criterion that a given set of elements of the free operad forms a Gr\"obner basis, and the Buchberger algorithm for constructing the Gr\"obner basis.

The following steps in the Gr\"obner bases theory are absolutely analogous to the case of the bases in the free associative algebras.

If the tree $w$ is divisible by $v$, there is a sequence of shuffle compositions which could be applied to $v$ to obtain $w$, denote it by $m_{w,v}$. It is possible to apply the same sequence of shuffle compositions to any other tree $z$ with the same set of inputs as in $v$. The result of this operation we denote by $m_{w,v}(z)$.

If we have two tree-polynomials (linear combination of tree-monomials) $f$ and $g$ and the highest tree-monomials of them divide one another: $\bar f$ is divisible by $\bar g$, then we define a result of
a {\it reduction} of $f$ by $g$, as the polynomial

$$ r_g(f)= f - \frac{c_f}{c_g} m_{\bar f,\bar g}(g).$$

The reduction has the property that the resulting polynomial is smaller then the initial one, due to the admissibility of the ordering w.r.t. the composition product.

\begin{definition}
The Gr\"obner basis of the  operadic ideal $I$ is the set of tree-polynomials, such that any element of $I$ can be reduced to zero by elements of this set.
\end{definition}

To give a criterion, which allows to check whether the given set of tree-polynomials is a  Gr\"obner basis, we need two more definitions.

\begin{definition}
If the monomial is divisible by two other monomials, it is called  {\it an ambiguity}.
\end{definition}

It means the monomial allows two different applications of reduction. The polynomial which we get if we subtract results of those two different applications of reductions is called $s$-{\it polynomial} corresponding to the ambiguity.

{\bf Criterion} If all ambiguities formed by the set of tree-polynomials are solvable (or there are no ambiguities), which means that corresponding  $s$-{\it polynomials} can be reduced to zero, then this set of tree-polynomials forms a Gr\"obner  basis of the ideal generated by them.

The Buchberger algorithm for construction of the Gr\"obner basis, from the given set of tree-polynomials
also works exactly like in algebra case.

We will use the following Priddy's criterion of Koszulity, which in operadic setting is explained in \cite{Bruno}.

\begin{theorem}\label{Pr}
If an operad $\cal P$
 has a presentation by generators and relations with a quadratic Gr\"obner basis (basis of degree two), then it is Koszul.
\end{theorem}

\section{Deformations of the operad {\cal A}ss}\label{Def}

In the paper  \cite{FG} there is a list of 10 a priori possibilities to introduce the notion of Hom--associative algebra. Let us have an unary linear operation $\alpha : A \to A$ on the linear space of an algebra $A$. A priori, we can consider  various generalised versions of associativity, defined w.r.t. $\alpha$. The list of possible identities \cite{FG} is the following:

\vspace{5mm}

$$ I_1: \,\, (ab)\a(c)=\a(a)bc$$
$$ I_2: \,\,  (a \a(b))c=a(\a(b)c)$$
$$ I_3: \,\,  (\a(a)b)c=a(b\a(c))$$

$$ II_0: \,\,  (\a(ab))c=a(\a(bc))$$
$$ II_1: \,\,  (\a(a)\a(b))c=a(\a(b)\a(c))$$
$$ II_2: \,\,  (\a(a)b))\a(c)=\a(a)(b\a(c))$$
$$ II_3: \,\,  (a\a(b))\a(c)=\a(a)(\a(b)c)$$

$$ III: \,\,  \a((ab)c)=\a(a(bc))$$
$$ III': \,\,  \a(ab))\a(c)=\a(a)(\a(bc))$$
$$ III'': \,\,  (\a(a)\a(b))\a(c)=\a(a)(\a(b)\a(c))$$

\vskip1cm

The corresponding operads, which carry information on the identities $I_0 $ --- $III''$ we  denote
${\cal O}_{I_0},\cdots,{\cal O}_{III''}$ respectively. Our goal is to classify them with respect to the confluence property of their natural presentation by generators and relations, as a nonsymmetric operad. Those presentations are as follows. Generators of these operads are the unary operation $\a$, denoted by an empty dot, and the binary operation $\mu$ denoted by a filled dot:
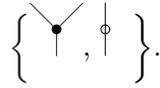
\setlength{\unitlength}{0.7mm}
$$\bigg\{
\begin{picture}(10,20)(0,15)
\drawline(5,15)(5,20)
\drawline(5,20)(10,25)
\drawline(5,20)(0,25)
\put(5,20){\circle*{1.7}}
\end{picture},\
\begin{picture}
(5,5)(5,-5)
\drawline(5,-5)(5,0)
\drawline(5,0)(5,5)
\put(5,0){\circle{1.7}}
\end{picture}
\bigg\}.
$$

The first one corresponds to the multiplication $\mu$ and the second to the linear map $\alpha$.

The relations are respectively:

$I_1$
\begin{center}
\setlength{\unitlength}{1.0mm}
\begin{picture}
(25,15)(5,5)
\drawline(15,0)(15,5)
\drawline(15,5)(20,10)
\drawline(15,5)(10,10)
\drawline(20,10)(25,15)
\drawline(10,10)(5,15)
\drawline(10,10)(15,15)
\put(15,5){\circle*{1.7}}
\put(20,10){\circle{1.7}}
\put(10,10){\circle*{1.7}}
\end{picture}
-
\begin{picture}
(25,15)(5,5)
\drawline(15,0)(15,5)
\drawline(15,5)(20,10)
\drawline(15,5)(10,10)
\drawline(20,10)(25,15)
\drawline(20,10)(15,15)
\drawline(10,10)(5,15)
\put(15,5){\circle*{1.7}}
\put(20,10){\circle*{1.7}}
\put(10,10){\circle{1.7}}
\end{picture},
\end{center}

$I_2$
\begin{center}
\setlength{\unitlength}{0.8mm}
\begin{picture}
(35,20)(5,5)
\drawline(20,0)(20,5)
\drawline(20,5)(15,10)(10,15)(5,20)
\drawline(20,5)(25,10)(30,15)(35,20)
\drawline(15,10)(20,15)(25,20)
\put(20,5){\circle*{1.7}}
\put(15,10){\circle*{1.7}}
\put(20,15){\circle{1.7}}
\end{picture}
-
\begin{picture}
(35,20)(5,5)
\drawline(20,0)(20,5)
\drawline(20,5)(15,10)(10,15)(5,20)
\drawline(20,5)(25,10)(30,15)(35,20)
\drawline(25,10)(20,15)(15,20)
\put(20,5){\circle*{1.7}}
\put(25,10){\circle*{1.7}}
\put(20,15){\circle{1.7}}
\end{picture}
\end{center}

$I_3$
\begin{center}
\setlength{\unitlength}{0.8mm}
\begin{picture}
(35,20)(5,5)
\drawline(20,0)(20,5)
\drawline(20,5)(15,10)(10,15)(5,20)
\drawline(20,5)(25,10)(30,15)(35,20)
\drawline(15,10)(20,15)(25,20)
\put(20,5){\circle*{1.7}}
\put(15,10){\circle*{1.7}}
\put(10,15){\circle{1.7}}
\end{picture}
-
\begin{picture}
(35,20)(5,5)
\drawline(20,0)(20,5)
\drawline(20,5)(15,10)(10,15)(5,20)
\drawline(20,5)(25,10)(30,15)(35,20)
\drawline(25,10)(20,15)(15,20)
\put(20,5){\circle*{1.7}}
\put(25,10){\circle*{1.7}}
\put(30,15){\circle{1.7}}
\end{picture}
\end{center}

$II_0$
\begin{center}
\setlength{\unitlength}{0.8mm}
\begin{picture}
(35,20)(5,5)
\drawline(20,0)(20,5)
\drawline(20,5)(15,10)(10,15)(5,20)
\drawline(20,5)(25,10)(30,15)(35,20)
\drawline(10,15)(15,20)
\put(20,5){\circle*{1.7}}
\put(15,10){\circle{1.7}}
\put(10,15){\circle*{1.7}}
\end{picture}
-
\begin{picture}
(35,20)(5,5)
\drawline(20,0)(20,5)
\drawline(20,5)(15,10)(10,15)(5,20)
\drawline(20,5)(25,10)(30,15)(35,20)
\drawline(30,15)(25,20)
\put(20,5){\circle*{1.7}}
\put(25,10){\circle{1.7}}
\put(30,15){\circle*{1.7}}
\end{picture}
\end{center}

$II_1$
\begin{center}
\setlength{\unitlength}{0.8mm}
\begin{picture}
(35,20)(5,5)
\drawline(20,0)(20,5)
\drawline(20,5)(15,10)(10,15)(5,20)
\drawline(20,5)(25,10)(30,15)(35,20)
\drawline(15,10)(20,15)(25,20)
\put(20,5){\circle*{1.7}}
\put(15,10){\circle*{1.7}}
\put(20,15){\circle{1.7}}
\put(10,15){\circle{1.7}}
\end{picture}
-
\begin{picture}
(35,20)(5,5)
\drawline(20,0)(20,5)
\drawline(20,5)(15,10)(10,15)(5,20)
\drawline(20,5)(25,10)(30,15)(35,20)
\drawline(25,10)(20,15)(15,20)
\put(20,5){\circle*{1.7}}
\put(25,10){\circle*{1.7}}
\put(20,15){\circle{1.7}}
\put(30,15){\circle{1.7}}
\end{picture}
\end{center}

$II_2$
\begin{center}
\setlength{\unitlength}{0.8mm}
\begin{picture}
(35,20)(5,5)
\drawline(20,0)(20,5)
\drawline(20,5)(15,10)(10,15)(5,20)
\drawline(20,5)(25,10)(30,15)(35,20)
\drawline(15,10)(20,15)(25,20)
\put(20,5){\circle*{1.7}}
\put(15,10){\circle*{1.7}}
\put(10,15){\circle{1.7}}
\put(25,10){\circle{1.7}}
\end{picture}
-
\begin{picture}
(35,20)(5,5)
\drawline(20,0)(20,5)
\drawline(20,5)(15,10)(10,15)(5,20)
\drawline(20,5)(25,10)(30,15)(35,20)
\drawline(25,10)(20,15)(15,20)
\put(20,5){\circle*{1.7}}
\put(25,10){\circle*{1.7}}
\put(15,10){\circle{1.7}}
\put(30,15){\circle{1.7}}
\end{picture}
\end{center}

$II_3$
\begin{center}
\setlength{\unitlength}{0.8mm}
\begin{picture}
(35,20)(5,5)
\drawline(20,0)(20,5)
\drawline(20,5)(15,10)(10,15)(5,20)
\drawline(20,5)(25,10)(30,15)(35,20)
\drawline(15,10)(20,15)(25,20)
\put(20,5){\circle*{1.7}}
\put(15,10){\circle*{1.7}}
\put(20,15){\circle{1.7}}
\put(25,10){\circle{1.7}}
\end{picture}
-
\begin{picture}
(35,20)(5,5)
\drawline(20,0)(20,5)
\drawline(20,5)(15,10)(10,15)(5,20)
\drawline(20,5)(25,10)(30,15)(35,20)
\drawline(25,10)(20,15)(15,20)
\put(20,5){\circle*{1.7}}
\put(25,10){\circle*{1.7}}
\put(20,15){\circle{1.7}}
\put(15,10){\circle{1.7}}
\end{picture}
\end{center}

$III$
\begin{center}
\setlength{\unitlength}{0.8mm}
\begin{picture}
(25,20)(5,5)
\drawline(15,0)(15,5)(15,10)
\drawline(15,10)(10,15)(5,20)
\drawline(15,10)(20,15)(25,20)
\drawline(10,15)(15,20)
\put(15,10){\circle*{1.7}}
\put(10,15){\circle*{1.7}}
\put(15,5){\circle{1.7}}
\end{picture}
-
\begin{picture}
(25,20)(5,5)
\drawline(15,0)(15,5)(15,10)
\drawline(15,10)(10,15)(5,20)
\drawline(15,10)(20,15)(25,20)
\drawline(20,15)(15,20)
\put(15,10){\circle*{1.7}}
\put(20,15){\circle*{1.7}}
\put(15,5){\circle{1.7}}
\end{picture}
\end{center}

$III'$
\begin{center}
\setlength{\unitlength}{0.8mm}
\begin{picture}
(35,20)(5,5)
\drawline(20,0)(20,5)
\drawline(20,5)(15,10)(10,15)(5,20)
\drawline(20,5)(25,10)(30,15)(35,20)
\drawline(10,15)(15,20)
\put(20,5){\circle*{1.7}}
\put(15,10){\circle{1.7}}
\put(10,15){\circle*{1.7}}
\put(25,10){\circle{1.7}}
\end{picture}
-
\begin{picture}
(35,20)(5,5)
\drawline(20,0)(20,5)
\drawline(20,5)(15,10)(10,15)(5,20)
\drawline(20,5)(25,10)(30,15)(35,20)
\drawline(30,15)(25,20)
\put(20,5){\circle*{1.7}}
\put(25,10){\circle{1.7}}
\put(30,15){\circle*{1.7}}
\put(15,10){\circle{1.7}}
\end{picture}
\end{center}

$III''$
\begin{center}
\setlength{\unitlength}{0.8mm}
\begin{picture}
(35,20)(5,5)
\drawline(20,0)(20,5)
\drawline(20,5)(15,10)(10,15)(5,20)
\drawline(20,5)(25,10)(30,15)(35,20)
\drawline(15,10)(20,15)(25,20)
\put(20,5){\circle*{1.7}}
\put(15,10){\circle*{1.7}}
\put(20,15){\circle{1.7}}
\put(10,15){\circle{1.7}}
\put(30,15){\circle{1.7}}
\end{picture}
-
\begin{picture}
(35,20)(5,5)
\drawline(20,0)(20,5)
\drawline(20,5)(15,10)(10,15)(5,20)
\drawline(20,5)(25,10)(30,15)(35,20)
\drawline(25,10)(20,15)(15,20)
\put(20,5){\circle*{1.7}}
\put(25,10){\circle*{1.7}}
\put(20,15){\circle{1.7}}
\put(30,15){\circle{1.7}}
\put(10,15){\circle{1.7}}
\end{picture}
\end{center}



\vskip17mm

\begin{theorem} Only one Hom--deformation of the associative operad ${\cal A}ss$ is confluent, namely the one given by relation $II_0$. 
\end{theorem}

\begin{proof}

Applying the Gr\"obner basis theory  for nonsymmetric operads \cite{Hof, Dots, Bruno}, to these presentations, we can check that only one of them satisfies the confluence property:

\vspace{5 mm}

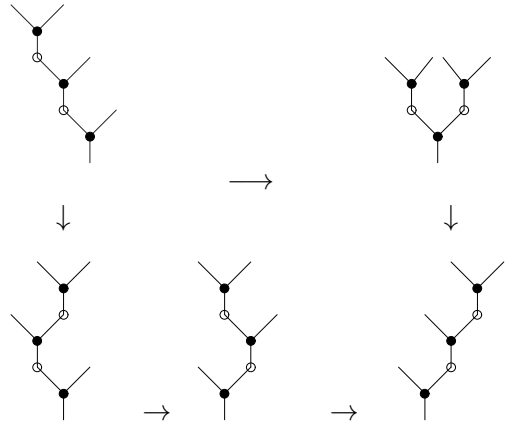
\[\begin{array}{ccccc}
\begin{picture}(20,30)(5,-5)
\drawline(20,0)(20,5)(15,10)(15,15)(10,20)(10,25)(5,30)
\drawline(10,25)(15,30)
\drawline(15,15)(20,20)
\drawline(20,5)(25,10)
\put(20,5){\circle*{1.7}}
\put(15,15){\circle*{1.7}}
\put(10,25){\circle*{1.7}}
\put(15,10){\circle{1.7}}
\put(10,20){\circle{1.7}}
\end{picture}
 & \  & \longrightarrow & \ &
  \begin{picture}(25,20)(5,-5)
\drawline(15,0)(15,5)(10,10)(10,15)(5,20)
\drawline(15,5)(20,10)(20,15)(25,20)
\drawline(10,15)(14,20)
\drawline(20,15)(16,20)
\put(15,5){\circle*{1.7}}
\put(10,15){\circle*{1.7}}
\put(20,15){\circle*{1.7}}
\put(10,10){\circle{1.7}}
\put(20,10){\circle{1.7}}
\end{picture}
  \\
 \downarrow & \  &\   &\   & \downarrow   \\
 \begin{picture}(20,35)(5,0)
\drawline(15,0)(15,5)(10,10)(10,15)(5,20)
\drawline(10,15)(15,20)(15,25)(10,30)
\drawline(15,25)(20,30)
\drawline(15,5)(20,10)
\put(15,5){\circle*{1.7}}
\put(10,15){\circle*{1.7}}
\put(15,25){\circle*{1.7}}
\put(10,10){\circle{1.7}}
\put(15,20){\circle{1.7}}
\end{picture}
  & \rightarrow  &
   \begin{picture}(20,35)(5,0)
\drawline(10,0)(10,5)(5,10)
\drawline(10,5)(15,10)(15,15)(10,20)(10,25)(5,30)
\drawline(10,25)(15,30)
\drawline(15,15)(20,20)
\put(10,5){\circle*{1.7}}
\put(15,15){\circle*{1.7}}
\put(10,25){\circle*{1.7}}
\put(15,10){\circle{1.7}}
\put(10,20){\circle{1.7}}
\end{picture}
& \rightarrow &
\begin{picture}(20,35)(5,0)
\drawline(10,0)(10,5)(5,10)
\drawline(10,5)(15,10)(15,15)(10,20)
\drawline(15,15)(20,20)(20,25)(15,30)
\drawline(20,25)(25,30)
\put(10,5){\circle*{1.7}}
\put(15,15){\circle*{1.7}}
\put(20,25){\circle*{1.7}}
\put(15,10){\circle{1.7}}
\put(20,20){\circle{1.7}}
\end{picture}
\end{array}
\]

\vspace{5 mm}


Using the analogue of the Priddy criterion for nonsymmertic operads we deduce that ${\cal O}_{II_0}$ is Koszul.

Let us give an example of the case where the confluence does not hold. Consider, say, an ambiguity formed by relation $I_2$.


\vspace{5 mm}

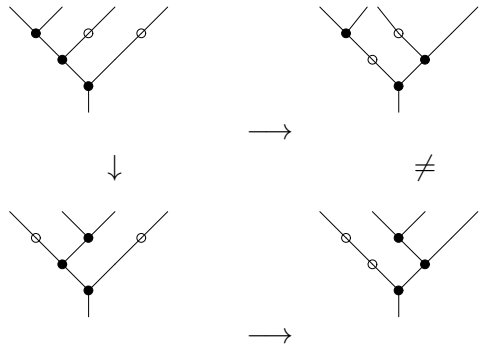
\[
\begin{array}{ccc}
\begin{picture}(40,30)(5,-5)
\drawline(20,0)(20,5)(15,10)(10,15)(5,20)
\drawline(20,5)(25,10)(30,15)(35,20)
\drawline(15,10)(20,15)(25,20)
\drawline(10,15)(15,20)
\put(20,5){\circle*{1.7}}
\put(15,10){\circle*{1.7}}
\put(10,15){\circle*{1.7}}
\put(20,15){\circle{1.7}}
\put(30,15){\circle{1.7}}
\end{picture}
  & \longrightarrow &
 \begin{picture}(40,30)(5,-5)
\drawline(20,0)(20,5)(15,10)(10,15)(5,20)
\drawline(20,5)(25,10)(30,15)(35,20)
\drawline(25,10)(20,15)(16,20)
\drawline(10,15)(14,20)
\put(20,5){\circle*{1.7}}
\put(25,10){\circle*{1.7}}
\put(10,15){\circle*{1.7}}
\put(15,10){\circle{1.7}}
\put(20,15){\circle{1.7}}
\end{picture}
  \\
 \downarrow & \       &  \neq  \\
 \begin{picture}(40,30)(5,-5)
\drawline(20,0)(20,5)(15,10)(10,15)(5,20)
\drawline(20,5)(25,10)(30,15)(35,20)
\drawline(15,10)(20,15)(25,20)
\drawline(20,15)(15,20)
\put(20,5){\circle*{1.7}}
\put(15,10){\circle*{1.7}}
\put(10,15){\circle{1.7}}
\put(20,15){\circle*{1.7}}
\put(30,15){\circle{1.7}}
\end{picture}
  & \longrightarrow  &
\begin{picture}(40,30)(5,-5)
\drawline(20,0)(20,5)(15,10)(10,15)(5,20)
\drawline(20,5)(25,10)(30,15)(35,20)
\drawline(25,10)(20,15)(15,20)
\drawline(20,15)(25,20)
\put(20,5){\circle*{1.7}}
\put(15,10){\circle{1.7}}
\put(10,15){\circle{1.7}}
\put(20,15){\circle*{1.7}}
\put(25,10){\circle*{1.7}}
\end{picture}
\end{array}
\]

\vspace{5 mm}



The verification of the confluence for all other operads in the list shows that the confluence does not hold for them.

\end{proof}

\section{Hom Gelfand-Dorfman algebras}\label{HGD}

In this paragraph we define a Hom--deformation of the Gelfand-Dorfman (GD) algebra, which have been defined, and its origin discussed in the introduction.















We define a Hom--GD algebra as follows. It is a quadruple $(A,\circ,[,],\a)$, where $A$ is a linear space, $\circ$ and $[,]$ are bilinear operations from $A \times A$ to $A$, and $\a$ is a linear space homomorphism from $A$ to $A$. They should satisfy the following identities:

\vspace{5mm}

$(1)^* \,\, (a\circ b)\circ \a(c)-\a(a)\circ(b\circ c)=(b\circ a)\circ\a(c)-\a(b)\circ(a\circ c) \,\,$ (Hom--left-symmetric)

\vspace{2mm}

$(2)^* \,\, (a\circ b)\circ\a(c)=(a\circ c)\circ\a(b) \,\,$  (Hom--right Novikov)

\vspace{2mm}

$(3)^* \,\, [a,b]=-[b,a] \,\,$

\vspace{2mm}

$(4)^* \,\, [[a,b],\a(c)]+[[c,a],\a(b)]+[[b,c],\a(a)]=0 \,\,$

\vspace{2mm}

$(5)^* \,\, [c\circ a,\a(b)]-[c\circ b,\a(a)]+[c,a]\circ\a(b)-[c,b]\circ\a(a)-\a(c)\circ[a,b]=0 \,\,$

\vspace{5mm}

Since GD algebra is a Lie algebra w.r.t. the operation $[,]$ and a Novikov algebra w.r.t. the operation $\circ$, a Hom--GD algebra will be naturally Hom--Lie \cite{MS} and Hom--Novikov \cite{Yau}, but with the additional compatibility condition between two operations and $\a$, defined by $(5)^*$.

First, note that for the defined in such a way Hom--GD algebras the following natural connection with Hom--Novikov algebras holds.

\begin{theorem}
The algebra $(A, \circ_{\a}, [\circ]_{\a})$ with two  operations, one of which is Hom--Novikov, i.e. satisfies the identities $(1)^*,(2)^*$, and anotheris  obtained from it as $[a,b]_{\a}=a \circ_{\a} b - b \circ_{\a} a$, is a Hom--GD algebra.
\end{theorem}

\begin{proof} It is known that structure obtained in  such a way is Hom-Novikov, and of course from the Hom left-symmetric identity $(1)^*$ it follows that the associated Lie bracket $[\circ]_{\a}$ satisfies the Hom Jacobi identity. So we need to check only that $(5^*)$ is satisfied.
Indeed,
$$ [c\circ a,\a(b)]_{\a}-[c\circ b,\a(a)]_{\a}+[c,a]_{\a} \circ\a(b)-[c,b]_{\a} \circ\a(a)-\a(c)\circ[a,b]_{\a}= \,\,$$

$$(c\circ_{\a} a)\circ_{\a}\a(b) -\a(b)\circ_{\a}(c\circ_{\a} a)-(c\circ_{\a} b)\circ_{\a}\a(a)+\a(a)\circ_{\a}(c\circ_{\a} b)+$$

$$(c\circ_{\a}a-a\circ_{\a}c)\circ_{\a}\a(b) -(c\circ_{\a}b-b\circ_{\a}c)\circ_{\a}\a(a) -\a(c)\circ_{\a}(a\circ_{\a}b-b\circ_{\a}a).$$

Clearly, the right-hand side is canceled due to $(1^*)$ and $(2^*)$.

\end{proof}

We can consider a particular case of Hom--GD algebra defined by $\a: A \to A$, which is a homomorphism of GD algebras. Namely, we can introduce a new algebra with operations  $\circ_{\a}$ and $[,]_{\a}$,
where $a \circ_{\a} b = \a(a\circ b)$ and $[a,b]_{\a}=\a(a\circ b)$. We prove below, that this is indeed Hom--GD algebra.

\begin{theorem} Let $(A, \circ, [,])$ be a GD algebra and $\a$ a GD algebra morphism, then
$A_{\a}=(A, \circ_{\a}, [,]_{\a})$, where $a \circ_{\a} b = \a(a\circ b)$ and $[a,b]_{\a}=\a([a, b]_{\a})$, is a Hom--GD algebra.
\end{theorem}

\begin{proof} Since $\a$ is a GD algebra morphism, i.e. $\a(a\circ b)=\a(a)\circ \a(b)$ and
 $\a([a,b])=[\a(a),\a b]$,  it is easy to verify, that

 \vspace{5mm}

 $1.\,\, (a\circ_{\a} b)\circ_{\a} \a(c)=\a^2((a\circ b)\circ c)$,

 \vspace{2mm}

$2.\,\,[[a, b]_{\a}, \a(c)]_{\a}=\a^2([[a,b],c])$,

\vspace{2mm}

 $3.\,\,[a, b]_{\a} \circ_{\a} \a(c)=\a^2([a,b]\circ  c)$,

 \vspace{2mm}

 $4.\,\,[a\circ_{\a} b, c]_{\a}=\a^2([a,b]\circ c)$.

 \vspace{5mm}

 Indeed, for example,

 $$[a, b]_{\a} \circ_{\a} \a(c)= \a( \a([a,b]) \circ \a(c))  =       \a^2([a,b]\circ  c),$$

 The first equality allows to prove identities $(1^*)$ and $(2^*)$ for $A_{\a}$, the second one could be used to verify the Jacobi identity $(3^*)$. The last two guarantee identity $(5^*)$. The anticommutativity of $[,]_{\a}$ is immediate.

\end{proof}

In terms of deformations of an operad, this special case of deformations has the feature, that  set of relations of the deformed operad  ${\cal A}_h$ contains as a subset relations of ${\cal A}$, which is of course not the case in general. In general we have  ${\cal A}=\langle x \,|\, F (x)\rangle $, and
${\cal A}_h=\langle x,h \,|\,\tilde F (x,h)\rangle$ such that
$\tilde F (x,h)$ coincide with $F(x)$, when $h$ is trivial. In the case when $\a$ is an algebra homomorphism,
we have ${\cal A}=\langle x \,|\, F (x), G(x,h)\rangle$, where $G(x, id)=F(x)$.


\section{Conformal Lie operad}\label{CL}


We will study here an operad ${\cal CL}ie$, associated to Lie conformal algebra
structure, and derive its Koszulity. There is another way to describe this structure (defined in the introduction),
by taking their generating function, called
the lambda bracket structure.
We let for any $a,b \in R$

$$[a_{\lambda} b] = \sum\limits_{n=0}^{\infty}\frac{\l^n}{n!} a_{(n)} b.
$$
Then we get an equivalent definition of a Lie conformal algebra \cite{Kac},
\cite{KA}:
\begin{definition}
We define a binary operation, called {\it $\lambda$ -- bracket}, on the $\mathbb{C}[\partial]$ module $R$: $\,R \times R \to \mathbb{C}[\partial, \l] \otimes_{\mathbb{C}[\partial]} \R : \, a \otimes b \to [a_{\lambda} b]$, as an operation  satisfying  the following axioms.
\end{definition}

I.{\it (sesquilinearity)}

$$[\partial a_{\lambda} b]=-\lambda [a_{\lambda} b]$$

$$[a_{\lambda} \partial b]= (\partial +\lambda)[a_{\lambda} b]$$

II. {\it (antisymmetry)}

$$[b_{\lambda} a]=-[a_{-\lambda-\partial} b]$$

III. {\it (generalised Jacobi identity)}

$$[a_{\lambda}[b_{\mu} c]]-[b_{\mu}[a_{\lambda} c]]=[[a_{\lambda} b]_{\lambda+\mu} c]$$

The Gelfand-Dorfman algebras considered above are closely related to the Lie
conformal algebras, as it was explained in the introduction. Here we give a formal proof.
Namely, having at hands Gelfand-Dorfman algebra with two operations $(\circ, [,])$ we can construct the $\lambda$ bracket as follows:

\begin{proposition} The binary operation

$$[a_{\lambda} b] = [a,b] + \partial (a\circ b) + \lambda (a\circ b+ b\circ a)$$

satisfies axioms of $\lambda$ bracket if and only if the operations $(\circ, [,])$  form a Gelfand-Dorfman algebra.

\end{proposition}

\begin{proof}

 We need to check that the above $\lambda$-bracket satisfies axioms

 II. {\it (antysymmetry)}

$$[b_{\lambda} a]=-[a_{-\lambda-\partial} b]$$

and

III. {\it (generalised Jacobi identity)}

$$[a_{\lambda}[b_{\mu} c]]-[b_{\mu}[a_{\lambda} c]]=[[a_{\lambda} b]_{\lambda+\mu} c]$$

if and only if for all $a,b$ the multiplication is

right symmetric:

$$ ass(a,b,c)=ass(a,c,b)\, : \, ass(a,b,c)=(ab)c-a(bc)$$

left Novikov:

$$a(bc)=b(ac)$$

and left Gelfand-Dorfman:

$$[ac,b]-[bc,a]+b[c,a]-a[c,b]-[a,b]c=0.$$

To check I, we observe that

$$ \partial (ba)+ \lambda(ba+ab)=-\partial (ab)+(\lambda + \partial)(ab+ba).$$

To check II, each of three terms can be rewritten in the following way:

$$
[[a_{\lambda}, b]_{\l+\mu} c]=[[a,b]+\partial (ab)+\l(ab+ba)_{\l+\mu} \, c]=$$

$$[[a,b]_{\l+\mu} \, c]+[\partial (ab)_{\l+\mu} \, c]+ \l [(ab+ba)_{\l+\mu} \, c]=$$

$$-(\l+\mu)[ab_{\l+\mu} \, c]+ \l[(ab+ba)_{\l+\mu} \, c]=$$

$$-\mu[ab_{\l+\mu} \, c] + \l [ba_{\l+\mu} \, c] =$$

$$=\mu([a.b]c+\partial ((ab)c)+(\l+\mu)((ab)c+c(ab)))+$$

$$ \l ([b,a]c+\partial ((ba)c+(\l+\mu)((ba)c+c(ba))   ).$$

\vskip5mm

Analogously we rewrite $[a_{\lambda}[ b_{\mu} c]]$ and $[b_{\mu}[a_{\lambda} c]]$.

Then we can consider their sum as a polynomial  in $\partial , \mu, \l$, so it is zero iff  'coefficients' of
all monomials in the variables $\partial, \mu, \l$ are zero.

The fact that the coefficient of $\partial$ is equal to zero means that the left GD-identity is satisfied.

The coefficient of $\mu$ is:

$$-[ab,c]+[b,ca]+b[a,c]+[a,c]b+[a,b]c+c[a,b]-[a,bc+cb]$$

According to GD identity we can substitute

$$-[ab,c]= -[cb,a]+c[b,a]-a[b,c]-[a,c]b$$

into the previous expression, which after some cancelations gives us:

$$[bc,a]-[ac,b]=-a[c,b]+b[c,a]+[b,a]c,$$

which is again GD identity.

The same will happen with the coefficient of $\l$.

The coefficient of $\partial^2$ is equal to zero iff left Novikov identity is satisfied.

The  coefficients of $\mu^2, \l^2, \l\mu, \l\partial, \mu \partial$
are zero iff the right symmetric identity holds.

\end{proof}

This close connection between those two structures motivates us to look at the properties of the operad governing  conformal Lie algebras.

Let us define the operad ${\cal CL}ie$, which will actually be a family of operads ${\cal CL}ie_k$, parametrised by the positive integers.
The first axiom of the Lie conformal algebra says that big enough products are zero. Hence we define  operads ${\cal CL}ie_k, \,\, k\in \mathbb Z$, such that in the Lie conformal algebras over each of them  all operations $(n)$ for $n>k$ are zero.


\begin{definition}
The set of generating operations for the operad ${\cal CL}ie_k$ is formed by the operations
$\{a,b\}_{n,j}=\partial^{(j)}(a_{(n)}b)$, for $n,j=0,1,2,\cdots$.
The set of relations on this generators is the following:

$$(F3) \,\, \{a,b\}_{n,0} = -\sum\limits_{j=0}^{k} (-1)^{n+j} \{b,a\}_{n+j,j}$$

$$(F4) \,\, \{a,\{b,c\}_{n,0}\}_{m,0}- \{b,\{a,c\}_{m,0}\}_{n,0}=
\sum\limits_{j=0}^{m} \left({m \atop j}\right) \{\{a,b\}_{j,0},c\}_{n+m-j,0}.$$

\end{definition}

\vspace{8 mm}

\begin{theorem}\label{CLKosz}
The Conformal Lie operad ${\cal CL}ie$ is Koszul.
\end{theorem}

\begin{proof} To prove Koszulity we pass to the corresponding  shuffle operad ${\cal CL}ie_{sh}$
and construct a Gr\"obner bases there, according to Section \ref{Gr}. The important step will be to choose a suitable ordering of tree-monomials. Then we will apply the Priddy criterion of Koszulity (see Theorem\ref{Pr}).

The shuffle operad   is generated by elements $ \{a,b\}_{n,j}, \,\, a<b$.
The presentation of those operations by tree-monomials looks as follows:

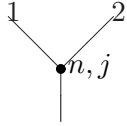
\vspace{5mm}

 \begin{center}
\begin{picture}
(25,15)(5,5)
\drawline(15,0)(15,5)
\drawline(15,5)(20,10)
\drawline(15,5)(10,10)
\drawline(20,10)(25,15)
\drawline(10,10)(5,15)
\put(25,15){\makebox(2,2){2}}
\put(5,15){\makebox(2,2){1}}
\put(15,5){\circle*{1.7}}
\put(17,4){\makebox(7,3){$n,j$}}
\drawline(15,0)(15,-5)
\end{picture}
\end{center}

 \vspace{15 mm}

So the relations are obtained by substitution of $(F3)$ to $(F4)$,
whenever we need to change the order of inputs.

The operad  ${\cal CL}ie_{sh}$ is given by relations:

$$(G4) \,\, \{a,\{b,c\}_{n,0}\}_{m,0}- \sum_{j=0}^k (-1)^{n+j} \{\{a,c\}_{m,0},b\}_{n+j,j}=
\sum\limits_{j=0}^{m}  \left({ m \atop j} \right)
\{\{a,b\}_{j,0},c\}_{n+m-j,0}$$

for any $a<b<c, \,\, n,m = 0,1,2,\cdots$.

In the language of tree-monomials the defining relations look as follows:

\vspace{15 mm}

\setlength{\unitlength}{1.0mm}
\begin{picture}
(25,15)(5,5)
\drawline(10,10)(15,5)
\drawline(15,5)(20,10)
\drawline(15,5)(10,10)
\drawline(20,10)(25,15)
\drawline(10,10)(5,15)
\drawline(15,5)(15,-5)
\drawline(15,15)(20,10)
\put(25,15){\makebox(2,2){3}}
\put(5,15){\makebox(2,2){1}}
\put(15,15){\makebox(2,2){2}}
\put(20,10){\circle*{1.7}}
\put(15,5){\circle*{1.7}}
\put(20,5){\makebox(2,2){$m,0$}}
\put(25,10){\makebox(2,2){$n,0$}}
\end{picture}
$\displaystyle -(-1)^{n+j} \sum\limits_{j=0}^{k}\,\,$
\setlength{\unitlength}{1.0mm}
\begin{picture}
(25,15)(5,5)
\drawline(10,10)(15,5)
\drawline(15,5)(20,10)
\drawline(15,5)(10,10)
\drawline(20,10)(25,15)
\drawline(10,10)(5,15)
\drawline(15,5)(15,-5)
\drawline(10,10)(15,15)
\put(25,15){\makebox(2,2){2}}
\put(5,15){\makebox(2,2){1}}
\put(15,15){\makebox(2,2){3}}
\put(10,10){\circle*{1.7}}
\put(15,5){\circle*{1.7}}
\put(19,4){\makebox(7,3){n+j,j}}
\put(4,8){\makebox(2,2){$m,0$}}
\end{picture}
$\displaystyle -\sum\limits_{j=0}^m  \left( {m \atop j} \right)\,\,$
\begin{picture}
(25,15)(5,5)
\drawline(10,10)(15,5)
\drawline(15,5)(20,10)
\drawline(15,5)(10,10)
\drawline(20,10)(25,15)
\drawline(10,10)(5,15)
\drawline(15,5)(15,-5)
\drawline(10,10)(15,15)
\put(25,15){\makebox(2,2){3}}
\put(5,15){\makebox(7,2){1}}
\put(15,15){\makebox(2,2){2}}
\put(10,10){\circle*{1.7}}
\put(15,5){\circle*{1.7}}
\put(22,4){\makebox(10,3){$m+n-j,0$}}
\put(5,8){\makebox(2,2){$j,0$}}
\end{picture}

\vspace{15 mm}

We will choose now an appropriate admissible ordering on the tree-monomials, which will lead to a quadratic Gr\"obner basis.

Let us take a path-lexicographic ordering defined in section \ref{Gr}, but we should decide how to
order operations, marked in our case by pairs $m,j$.  We order those  pairs lexicographically from right to left, i.e. $\{(m,i)\}>\{(n,j)\}$, whenever $i>j$, or $i=j$ and $m>n$. We suppose symbols themselves ordered as follows:
$0<2<3<\cdots<1 --$ all as numbers, but $1$ is bigger than others. In particular, the operation $(m,0)$ is always smaller then $(n,j), \, j\neq 0$ and $(m,i)$ is smaller then $(n,1)$ if $i\neq 1$.

Under the above ordering, when comparing tree-monomials in the relations, we conclude that the highest tree monomial has the shape

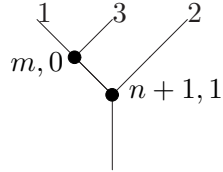
\begin{center}
\setlength{\unitlength}{1.0mm}
\begin{picture}
(25,15)(5,5)
\drawline(10,10)(15,5)
\drawline(15,5)(20,10)
\drawline(15,5)(10,10)
\drawline(20,10)(25,15)
\drawline(10,10)(5,15)
\drawline(15,5)(15,-5)
\drawline(10,10)(15,15)
\put(25,15){\makebox(2,2){2}}
\put(5,15){\makebox(2,2){1}}
\put(15,15){\makebox(2,2){3}}
\put(10,10){\circle*{1.7}}
\put(15,5){\circle*{1.7}}
\put(20,4){\makebox(7,3){$n+1,1$}}
\put(4,8){\makebox(2,2){$m,0$}}
\end{picture}
\end{center}

\vspace{15 mm}

This follows from the inequalities:

$$\{(m,0),(m,0)(n,0),(m,0)(n,0), (123)\} < \{(n+j,j)(m,0),(n+j,j),(n+j,j)(m,0), (132)\}$$

$$  \{(n+m-j,0)(j,0),(n+m-j,0)(j,0),(m+n-j,0), (123)\}  $$

   $$< \{(n+j,j)(m,0),(n+j,j),(n+j,j)(m,0), (132)\}.$$

We should find out now which ambiguities this set of highest words (when $m$ and $n$ are ranging over 0,1,2,..) can form. For this we need to prove the following

\begin{lemma} For arbitrary $n,m = 0,1,2,\cdots$ the operations $a_{(n)}b $ and
$\partial (a_{(m)}b) $ do not coincide.
\end{lemma}

\begin{proof} For any fixed $n$ we construct a free $\mathbb C[\partial]$ module $M_n$, together with operations  $a_{(n)}b $ on it, which satisfy axioms (c1)-(c4), but such that $a_{(n)}b \neq
\partial (a_{(m)}b) $ for any $m=0,1,2,\cdots$.
This means for any $n$ we will have a concrete conformal Lie algebra $M_n$, where $a_{(n)}b \neq
\partial (a_{(m)}b) $ for any $m$.  The statement of the lemma will follow then.

As it has been pointed out for example in the Remark(2.2) in \cite{KA}, if we define operations satisfying (c3) and (c4) on the basis of a free $\mathbb C[\partial]$-module $R$, they can be extended uniquely to the whole $R$, according to (c2). So we construct our examples having this in mind. All our modules $M_n$ will be free modules over $\mathbb C[\partial]$, with the basis consisting of three elements, which we denote $a,b,c$. So all what we need is to define operations satisfying (c3) and (c4) on $a,b,c$, such that  $x_{(n)}y \neq
\partial (x_{(m)}y) $ for any $n,m$ and  for some elements $x$ and $y$ in $M_n$.

Let us first construct appropriate modules for $n=0,1$ for illustration, and then give a general construction.

In $M_0$, define $a_{(0)}b=c, b_{(0)}a=-c$ and the rest of the products of elements from the basis $a,b,c$ are equal to zero:

 $a_{(0)}c=0, c_{(0)}a=0$,

$b_{(0)}c=0, c_{(0)}b=0$, and

$z_{(i)}w=0$ for all $i>0$, and $z,w \in M_0$.

Surely (c3) and (c4)are satisfied in this case. We choose $a_{(0)}b$ and $ b_{(0)}a$ related in such  a way, that  (c3) is satisfied, and (c4) is satisfied since with such a multiplications each product of three elements should be zero.

Now ensure that the lemma condition is satisfied for $n=0, m=0,1,2,\cdots$. Indeed,

$$c=a_{(0)}b \neq \partial (a_{(m)}b)= \partial(0)=0$$

for any $m=1,2,\cdots$, and

$$c=a_{(0)}b \neq \partial (a_{(m)}b)= \partial c$$

for $m=0$.

These hold true since $c$ is an element in the basis of the free $\mathbb C[\partial]$ module, so $c\neq 0$ and $c\neq \partial c$.

Take now $n=1$ and construct the module $M_1$. Suppose

 $$a_{(1)}b=c, b_{(1)}a=c.$$

 Then  (c3), considered as a triangular system of equations, which express operations $a_{(n)}b$  via operations with bigger numbers, will give us the equation

  $$a_{(0)}b= b_{(0)}a-\partial (b_{(1)}a).$$

  The possible solution is $a_{(0)}b= 0$ and $ b_{(0)}a=\partial c.$
The rest of operations on the pairs from $a,b,c$ we again put to be zero. The axiom (c4) is again automatically satisfied.

The  conditions of the lemma hold for $n=1$ and any $m$:

$$c=a_{(1)}b \neq \partial (a_{(m)}b)$$

since $\partial (a_{(m)}b)=0$ for $m=0$; $\partial (a_{(m)}b)=\partial c$  for $m=1$ and $\partial (a_{(m)}b)=0$ for  $m=2,3,\cdots$.

Describe now the general construction of $M_n$. Let $M_n$ be a free $\mathbb C[\partial]$ module with the basis $\{a,b,c\}$, endowed with the following operations:

$$a_{(i)}b=0, \,\,\, \text{ for any} \,\,\, i>n $$

$$a_{(n)}b =c, \,\,\,  b_{(n)}a=(-1)^{n+1}c$$

$$a_{(n-1)}b =0, \,\,\,  b_{(n-1)}a=(-1)^{n+1}\partial c$$

$$a_{(n-2)}b =0, \,\,\,  b_{(n-2)}a$$

$$\cdots$$

$$a_{(0)}b =0, \,\,\,  b_{(0)}a$$

All products of other pairs of elements of the basis are equal to zero.

We claim that

1). $a_{(n)}b \neq  \partial (a_{(m)}b)$ for any $m$. Indeed, $a_{(n)}b =c$ is not equal to one of the following:

$\partial (a_{(0)}b)=\partial (0)=0$

$\partial (a_{(1)}b)=\partial (0)=0$

$\cdots$

$\partial (a_{(n-2)}b)=\partial (0)=0$

$\partial (a_{(n-1)}b)=\partial (0)=0$

$\partial (a_{(n)}b)=\partial c$

The reason why $c \neq \partial c$ and $c \neq 0$ is again that c is an element of the basis of the free
$\mathbb C[\partial]$ module.

2). Axioms

 $$(c3) \,\,\, a_{(r)} b = -(-1)^{r+j} \sum\limits_{j=0}^{\infty} \partial^{(j)}( b_{(r+j)} a) $$

and (c4) are satisfied in $ M_n$  for arbitrary $r$.

Indeed, for $0\leq r \leq n$ the equality

$$ a_{(r)} b = -(-1)^{r+j} \sum\limits_{j=0}^{k} \partial^{(j)} (b_{(r+j)} a )=$$

 $$-((-1)^{n-1} \partial^{(n-r-1)}(b_{(n-1)}a) + (-1)^{n} \partial^{(n-r)}(b_{(n)}a)) = 0$$

 holds, when $ b_{(n-1)}a=(-1)^{n+1}\partial c$ and $b_{(n)}a=(-1)^{n+1}c$.

For $ r > n$ all terms of the axiom (c3) are zero, so it is automatically satisfied.

The axiom (c4) also holds true for such a multiplications since the product $(n)$ of any two elements of the basis is either zero or $c$. If it is zero, the product $(m)$ with the third element is also zero. If it is $c$, due to our definitions of the products, $c$, multiplied by any other element of the basis,  again give zero. So any combination of products of three elements from the basis is zero.

\end{proof}

{\bf Remark.} We can give another proof of this lemma, which rely on the construction of current Lie algebras. Namely, take the current Lie algebra ${\cal C}={\cal G} \otimes {\cal A}$, where ${\cal G}$
 is non-abelian Lie algebra. Define on ${\cal C}$ a Lie conformal structure in the following way:
 $ a_{(0)}b=[a,b]$, $a_{(n)}b=0, \, n=1,2,\cdots. a,b \in {\cal G} $. Then, since according to axiom (c2)
 $(\partial^{(n)} a)_{(m)}b=(-1)^n a_{(m-n)}b$, we can find two elements $A=\partial^{(n)} a, \,\, B=b$, for which
$A_{(n)}B \neq \partial A_{(m)}B$ for any $n,m$. Indeed, $(\partial^{(n)} a)_{(n)}b=(-1)^n a_{(0)}b$,
$\,\,\partial((\partial^{(n)} a)_{(m)}b)=(-1)^n \partial( a_{(m-n)}b)$, which is equal to $(-1)^n \partial (a_{(0)}b)$, if $m=n$and zero otherwise.

\vspace{5mm}

Since we have ensured, that the operations $a_{(n)}b $ and $\partial (a_{(m)}b) $  do not coincide for arbitrary $m$ and $n$, the highest words of the Jacobi relations  can not give rise to any ambiguity. The highest words of those relations w.r.t our ordering are of the shape
$\{\{a,c\}_{m,0},b\}_{n+1,1}$, depicted   above by the left comb labeled tree.
This means this presentation by generators and relations of the operad ${\cal CL}ie$,  with respect to the chosen ordering of tree-monomials, is formed by a combinatorially free set of quadratic (degree two) tree monomials.  As a consequence, quadratic relations form a Gr\"obner basis.

From this we can conclude that the operad is Koszul, due to the analogues of the Priddy criterion for operads (see Theorem \ref{Pr} in section \ref{Gr}).
\end{proof}

\section{Acknowledgements}

We are grateful to Victor Kac for his suggestion to look at the question of Koszulity of an operad associated to the axioms of Conformal Field Theory, and for many valuable discussions. We also would like to thank V. Dotsenko for a number of useful remarks. The work is supported by the Grant 9038 of the Estonian Scientific Council.

\normalsize

\vskip1truecm

\scshape

\noindent  Natalia Iyudu \\





\noindent Queens's University Belfast

\noindent Department of Pure Mathematics

\noindent University road, Belfast, BT7 1NN, UK

\noindent E-mail address: \qquad {\tt n.iyudu@qub.ac.uk}





\vskip 5mm

\noindent Abdenacer Makhlouf \\

\noindent Universit\'{e} de Haute Alsace,

\noindent Laboratoire de Math\'{e}matiques, Informatique et Applications,

\noindent 4, rue des Fr\`{e}res Lumi\`{e}re F-68093 Mulhouse, France

\noindent E-mail address: \qquad {\tt Abdenacer.Makhlouf@uha.fr}

\end{document}